\documentclass[12pt]{extarticle}
\usepackage{amsmath, amsthm, amssymb, color}
\usepackage[colorlinks=true,linkcolor=blue,urlcolor=blue]{hyperref}
\usepackage{graphicx}
\usepackage{caption}
\usepackage{mathtools}
\usepackage{enumerate}
\usepackage{verbatim}
\usepackage{tikz,tikz-cd,tikz-3dplot}
\usepackage{amssymb}
\usetikzlibrary{matrix}
\usetikzlibrary{arrows}
\usepackage{algorithm}
\usepackage[noend]{algpseudocode}
\usepackage{caption}
\usepackage[normalem]{ulem}
\usepackage{subcaption}
\usepackage{multirow}
\oddsidemargin \evensidemargin
\usepackage{makecell}
\usepackage{array}
\usepackage{enumitem}
\usepackage{xcolor}

\newtheorem{theorem}{Theorem}[section]
\newtheorem{proposition}[theorem]{Proposition}
\newtheorem{lemma}[theorem]{Lemma}
\newtheorem{corollary}[theorem]{Corollary}

\theoremstyle{definition}

\newtheorem{example}[theorem]{Example}

\newcommand{\PP}{\mathbb{P}}
\newcommand{\R}{\mathbb{R}}

\newcommand{\CC}{\mathbb{C}}

\newcommand{\St}[2]{\genfrac{\{}{\}}{0pt}{}{#1}{#2}}

\definecolor{LightGreen}{HTML}{4F7942}

\title{\bf Kinematic Stratifications}
\author{Veronica Calvo Cortes, Hadleigh Frost, and Bernd Sturmfels}

\date{}

\begin{document}

\maketitle

\begin{abstract}
\noindent
We study stratifications of regions in the space of symmetric matrices.
Their points are Mandelstam matrices for
momentum vectors in particle physics. Kinematic strata in these regions are indexed by signs and rank two matroids.
Matroid strata of Lorentzian quadratic forms
arise when all signs are non-negative.
 We characterize the posets of strata, 
 for massless and massive particles,
 with and without momentum conservation.
\end{abstract}
 
 \section{Introduction}
 
 In theoretical physics, the \emph{momentum} of a particle is a vector in  {\em Minkowski space}, the  real vector space $\R^{1+d}$ with the Lorentzian inner product
$$ \quad p \cdot q \,\, = \,\, 
    p_0q_0-p_1q_1-\cdots-p_d q_d, \quad \text{ where } \ p=(p_0,p_1,\ldots,p_d) \text{ and } q=(q_0,q_1,\ldots,q_d).    $$
    The {\em universal speed limit} states that the quadratic inequality
    $p \cdot p \geq 0$ holds for each particle. 
A particle is called {\em massless} if $p$ lies on the light cone, i.e.~if the equation $p \cdot p = 0$ holds. 

We consider configurations of $n$ particles, each represented
by its own momentum vector $p^{(i)} \in \R^{1+d}$, for $i=1,2,\dots,n$.
The Lorentz group ${\rm SO}(1,d)$ acts on such configurations. The kinematic data of the particles is invariant under this action. Such invariant quantities can be expressed using the {\em Mandelstam variables} $s_{ij}$, which are the entries of the Gram matrix
\begin{equation}
\label{eq:gram}
\begin{bmatrix} 
s_{11} & s_{12} & \cdots & s_{1n} \\
s_{12} & s_{22} & \cdots & s_{2n} \\
\vdots & \vdots & \ddots & \vdots \\
s_{1n} & s_{2n} & \cdots & s_{nn} \\
\end{bmatrix}
\,\,=\,\, \begin{small}
\begin{bmatrix}
\,- \,\,p^{(1)} \,- \,\\ \,-\,\,p^{(2)}\,- \,\\ 
\vdots \smallskip \\ \, -\,\,p^{(n)} \,-\,
\end{bmatrix} \end{small} \! \begin{small}
\begin{bmatrix} +1 & 0 &  \cdots & 0 \\
\,0  & \!\! -1 & \cdots & 0 \\
\,\vdots & \vdots  & \ddots & \vdots \\
\,0  & 0 &  \cdots &\! \! -1 \\
\end{bmatrix} \! 
\begin{bmatrix}
| & | & & \!|  \smallskip \\
(p^{(1)})^T \!& \!\!(p^{(2)})^T\! \!&\! \cdots &\!\! (p^{(n)})^T \smallskip \\
| & | &  & \!| \\
\end{bmatrix}\!. \end{small}
\end{equation}
This is a symmetric matrix of some rank $ r \leq d\!+\!1$. 
The variety of rank $r$ matrices
has codimension $\binom{n-r-1}{2}$ in the space of 
real symmetric $n\times n$ matrices. In view of the universal speed limit,
the equation (\ref{eq:gram}) parametrizes a subset of this variety. We call this subset
the {\em Mandelstam region}~$\mathcal{M}_{n,r}$. By quantifier elimination,~$\mathcal{M}_{n,r}$ is a semialgebraic set.
We are interested in the stratification of~$\mathcal{M}_{n,r}$ given
 by the signs of the matrix entries~$s_{ij}$.

 The intersection of the Mandelstam region with the cone of non-negative  matrices
 is of current interest in geometric combinatorics. We call this the {\em Lorentzian region},
 here denoted
  $$ \mathcal{L}_{n,r} \,\,\, = \,\,\,
 \mathcal{M}_{n,r} \,\,\cap\,\, (\R_{\geq 0})^{\binom{n+1}{2}} . $$
  The points in $\mathcal{L}_{n,\leq n} = \sqcup_{r=1}^n \mathcal{L}_{n,r} $
  are the {\em  Lorentzian polynomials} \cite{Branden, LorentzPol} of degree two,
  stratified as in \cite{BHKL}. Our article thus relates
  the geometry of  Lorentzian polynomials to scattering amplitudes \cite{BHPZ}.
  In our universe, particles satisfy momentum conservation (MC),
  and they may or may not be massless.
    We explore the combinatorial implications of these conditions.
    
We now present the organization of this paper, and we highlight our main results.
Section~\ref{sec2} gives the semialgebraic description
of the Mandelstam region $\mathcal{M}_{n,r}$
by alternating sign conditions on the principal minors of the matrix $S$.
Lemma \ref{lem:mand} is reminiscent of the familiar characterization of
positive definite matrices by the positivity of these minors.

Br\"anden \cite {Branden} proved that 
$\mathcal{L}_{n,\leq n}$ is a topological ball of dimension $\binom{n+1}{2}$.
In fact, $\mathcal{M}_{n,\leq n} = \sqcup_{r=1}^n \mathcal{M}_{n,r}$ is a disjoint union of $2^{n-1}$ such balls,
 intersecting along lower-dimensional boundaries.
We also discuss what happens when the particles are {\em on-shell}.
This means that the diagonal entries $s_{ii} = p^{(i)} \cdot p^{(i)}$
take on fixed prescribed values  $m_i^2$. In physics,
the $m_i$ are the masses of the $n$ particles.
The particles can be massive $(m_i >0)$ or  massless $(m_i=0)$.

\smallskip

In Section \ref{sec3} we turn to the {\em massless Mandelstam region} $\mathcal{M}_{n,r}^0$. Its non-negative
part~is the
{\em massless Lorentzian region} $\mathcal{L}_{n,r}^0$. These regions arise by intersecting
$\mathcal{M}_{n,r}$ and $\mathcal{L}_{n,r}$
with the linear subspace of symmetric $n \times n$ matrices with zeros
on the diagonal:
\begin{equation}\label{eq:diagonalszero}
 s_{11} \,=\, s_{22} \,=\, \cdots \,=\, s_{nn} \,\, = \,\, 0. \qquad 
 \end{equation}
 In the setting of  \cite{BHKL, Branden, LorentzPol}, points in
 $\mathcal{L}^0_{n,r}$ are multiaffine Lorentzian polynomials 
   $\ell_1^2 - \ell_2^2 - \cdots - \ell_r^2$, where
 $\ell_1,\ell_2,\ldots,\ell_r$ are linear forms in $n$ variables.
 The work of
 Br\"anden \cite{Branden} also shows that
 $\mathcal{L}_{n,\leq n}^0$ is a ball.
 We focus on the regions $\mathcal{L}^0_{n,r}$ of fixed rank $r$,
 and we stratify these   by rank two matroids.
 This is closely related to
  the decomposition in \cite{BHKL}.
We introduce {\em signed matroids} to describe the
stratification of the larger region $\mathcal{M}^0_{n,r}$.
The main result of Section~\ref{sec3} is Theorem~\ref{thm:massless}.
Corollary \ref{cor:countingMassless} gives formulas for the number of strata in any given dimension.

\smallskip

Section \ref{sec4} is devoted to the topology of the strata.
The inclusion relations are well behaved (Proposition \ref{prop:nicer}).
However, the strata have non-trivial topology. 
Theorem \ref{thm:topology} shows that the strata, with rank constraints relaxed,
 are homotopic to
configuration spaces for $n$ labeled points on the $(r-2)$-sphere.
For $r=3$, our kinematic strata are typically
 disconnected (Proposition \ref{prop:disconnected}).
 For $r=4$, Corollary \ref{cor:braid} leads us to the
  complex moduli space
 $M_{0,n}(\CC)$.

\smallskip

Section \ref{sec5} treats a more challenging variant which is relevant
for the real world, namely we study the MMC region.
Here, MMC stands for massless with momentum conservation.
The {\em MMC region} is the intersection of $\mathcal{M}^0_{n,r}$
with the linear subspace defined by the equations
\begin{equation}
\label{eq:momcon}
\qquad s_{i1} + s_{i2} + \cdots + s_{in} \,\, = \,\, 0
\qquad {\rm for} \quad i=1,2,\ldots,n. 
\end{equation}
These are $n$ independent linear constraints, equivalent to requiring
that $\,  p^{(1)} + p^{(2)} + \cdots +  p^{(n)}$ is the zero vector in $\R^{1+d}$.
Thus, in Section~\ref{sec5}, all row sums and column sums of the
matrix $S = [s_{ij}]$ are zero. The MMC region lives
in  $\R^{n(n-3)/2}$ and it has $2^{n-1}-n-1$ connected components.
Our main result (Theorem \ref{thm:MatroidMomentumConservation}) describes
the boundary structure and stratification. The section concludes with a detailed case
study of the MMC stratifications for $n=4$ and $n=5$.

In Section \ref{sec6} we place our findings into the context of 
particle physics. We 
discuss kinematic stratifications for massive particles,
and we offer an outlook towards future research.

\newpage

\section{The Mandelstam Region}
\label{sec2}

A symmetric $n \times n$ matrix $S = [s_{ij}]$ of rank $r$ is said to be a \emph{Mandelstam matrix} if
\begin{itemize}
\item the diagonal entries are non-negative, $s_{ii} \geq 0$ for $i=1,\ldots,n$; and \vspace{-0.2cm}
\item it has precisely one positive eigenvalue and $r-1$ negative eigenvalues.
\end{itemize}
We denote the set of all Mandelstam matrices of rank $r$ by ${\cal M}_{n,r}$. 
This is a semialgebraic set in $\mathbb{R}^{n+1 \choose 2}$, the space of all
symmetric matrices. The following is the Mandelstam analogue of the
familiar characterization
of positive semidefinite matrices in terms of principal minors.

\begin{lemma} \label{lem:mand}
A symmetric $n \times n$ matrix $S$ is Mandelstam if and only if
 \begin{equation}
 \label{eq:minors}
 \quad (-1)^{ |I| -1 }\, {\rm det}(S_I)  \,\,\geq \,\, 0 \qquad \hbox{for all}\,\,\,\, I \subseteq [n], 
 \end{equation}
 where ${\rm det}(S_I)$ are the principal minors of $S$. 
\end{lemma}
 \begin{proof} This follows from the general results in  \cite{LorentzPol}.
 We refer to Baker's exposition in~\cite{Baker}.
The key step is Cauchy's interlacing theorem \cite{Cauchy}. This states 
that the eigenvalues of $S_I$ interlace the eigenvalues of $S_J$ whenever
 $I \subset J$ and $|I| = |J|-1$. Hence, if $S_I$ has at most one positive
eigenvalue then so does $S_J$. But $S_J$ cannot have
all negative eigenvalues because its trace is non-negative. The lemma follows then inductively.
\end{proof}

The name of our matrices refers to the physicist Stanley Mandelstam (1928--2016) who is credited for introducing the variables $s_{ij}$ in the context of scattering amplitudes.
In \cite{Mandelstam} the role of $\mathcal{M}_{n,r}$ as a kinematic space is recognized.
A term more familiar to mathematicians might be ``Lorentzian matrices.''
These encode Lorentzian quadratic forms   \cite{Branden,LorentzPol}.
We here use the term {\em Lorentzian matrix} for
a Mandelstam matrix whose entries $s_{ij}$ are all non-negative.

\smallskip

Mandelstam matrices arise as
 Gram matrices of momentum vectors in $\mathbb{R}^{1+d}$ with 
the Lorentzian inner product. A non-zero momentum vector is any vector $p\in \mathbb{R}^{1+d}$ of the form
\begin{equation}\label{eq:plambda}
p \,=\, \lambda \,(1,x_1,\ldots,x_{d}),
\end{equation}
for some scalar $\lambda\neq 0$, and $x = (x_1,\ldots,x_{d})$ in the closed unit ball $\mathbb{B}^d = \{x \in \R^d : || x || \leq 1 \}$. 
Given $n$ momentum vectors, $p^{(i)}$, their Gram matrix $S = [s_{ij}]$ has entries $s_{ij} = p^{(i)} \cdot p^{(j)}$.
This is the matrix in (\ref{eq:gram}).
The entries of $S$ may now be written as 
\begin{equation}
s_{ij} \,\,= \,\,\lambda_i \lambda_j \! \left( 1 - \langle x^{(i)}, x^{(j)} \rangle \right)
\end{equation}
Here  $\cdot$ is the Lorentz inner product on $\R^{1+d}$ 
 and $\langle\,\,,\, \rangle$ is the Euclidean inner product on $\R^d$.

\begin{lemma} \label{lem:gram}
A symmetric $n \times n$ matrix $S$ is Mandelstam, i.e.~$S$ lies in the region
${\cal M}_{n,\leq 1+d}$, if and only if it is the Gram matrix of $n$ momentum 
vectors in $(1+d)$-dimensional spacetime.
\end{lemma}

\begin{proof}
Assume that $S$ has no zero rows or columns. For the only-if direction, take
a Mandelstam matrix $S$. By Lemma \ref{lem:mand} and diagonalization of symmetric matrices, it can be factorized
as in \eqref{eq:gram}. Namely, we write
$S = M D M^T$,
where $D = {\rm diag}(1,-1,-1,\ldots,-1)$. Let the row vectors of $M$ be 
$p^{(i)} = \lambda_i (1,x^{(i)})$, for some $x^{(i)} \in \mathbb{R}^d$ and $\lambda_i \neq 0$. 
As $s_{ii} = \lambda_i^2 (1 - ||x^{(i)}||^2)$ is non-negative, we conclude that $x^{(i)} \in \mathbb{B}^d$.
Thus, the $p^{(i)}$ are momentum vectors in $\R^{1+d}$.

For the if direction, suppose that $S$ is  the Gram matrix of $n$ non-zero momentum vectors 
$p^{(i)} =  \lambda_i (1,x^{(i)})$, with $x^{(i)} \in \mathbb{B}^d$. Consider any subset
$I \subset [n]$ of cardinality $m+1$. The signed $m$-dimensional volume of the convex
hull of $\{ x^{(i)}\,:\, i \in I\}$ in $\R^d$ is equal to
\[
    V \,=\, \frac{1}{(m+1)!}\, \det \begin{bmatrix}\,
        1 & x^{(1)} \\ \,\vdots & \vdots \\ \,1 & x^{(m+1)}
    \end{bmatrix} \,=\, \frac{(-1)^m}{(m+1)!} \, \det \begin{bmatrix}
        1 & \cdots & 1\\ -(x^{(1)})^T & \cdots &-(x^{(m+1)})^T
    \end{bmatrix}.
\]
Here, we take $I = \{1,\ldots,m+1\}$, after relabeling. The product of these two formulas is the
determinant of the matrix product. We see that this determinant has the desired sign:
\[
0 \,\,\leq \,\,\Bigl( \,\prod_{i \in I} \lambda_i \,\Bigr)^{\! 2}\, V^2\,\, =\,\, \left( \frac{1}{(m+1)!}\right)^{\!\! 2}\, (-1)^m \, {\rm det}(S_I).
\]
Therefore, by Lemma \ref{lem:mand}, the Gram matrix $S$ is Mandelstam.
\end{proof}

Let us now consider the possible sign patterns of the off-diagonal entries in a Mandelstam matrix $S$. Applying Lemma \ref{lem:mand} to the principal minors of size $2$ and $3$, we observe
\begin{equation}\label{eq:2-3minors}
s_{ii}s_{jj}  \leq s_{ij}^2 \quad  \text{ and } \quad 2 s_{ij} s_{ik} s_{jk} + s_{ii}s_{jj}s_{kk} \geq s_{ii} s_{jk}^2 +s_{jj} s_{ik}^2 +s_{kk} s_{ij}^2. 
\end{equation}
By combining these two inequalities for any distinct $i,j,k$, it follows that
\begin{equation}\label{eq:signcond}
s_{ij} s_{ik} s_{jk}\,\, \geq \,\,0.
\end{equation}
We can now split $[n]$ in two subsets: if $s_{ij}$ is positive we put $i,j$ in the same subset and if it is negative we put each in a different subset. Equation \eqref{eq:signcond} guarantees we can do this consistently. In other words, if $S$ has no zero entries, then there is a sign vector $\sigma \in \{-,+\}^n$ so that ${\rm sgn}(s_{ij}) = \sigma_i \sigma_j$.
If we view $S$ as a Gram matrix of momentum vectors, then $\sigma_i$ is the sign of 
the multiplier $\lambda_i$ in \eqref{eq:plambda}. 
We can fix $\sigma_1 = +$, so
there are $2^{n-1}$ allowable choices of sign patterns.
We define the \emph{signed Mandelstam region} ${\cal M}_{n,\sigma,r} $
to be the closure in  $ {\cal M}_{n,r}$ of the subset of Mandelstam matrices with no zero entries whose signs are determined by $\sigma$.
 
\begin{corollary}\label{cor:signs}
The Mandelstam region ${\cal M}_{n,r}$ is the union of the $2^{n-1}$ signed Mandelstam regions ${\cal M}_{n,\sigma,r}$,
and the relative interiors of these regions are pairwise disjoint. In symbols,
\begin{equation}
{\cal M}_{n,r} \,\,=\,\, \bigcup_{\sigma} \, {\cal M}_{n,\sigma,r}.
\end{equation}
 \end{corollary}

The {\em Lorentzian region} $\mathcal{L}_{n,r}$ is the set of all Mandelstam matrices with non-negative entries. 
 It is the region ${\cal M}_{n,\sigma,r}$ with $\sigma = (+,+,\ldots,+)$. 
Thus ${\cal L}_{n,r} $ is the set of Lorentzian $n\times n$ matrices of rank~$r$. 
Many facts about Lorentzian matrices can be extrapolated to any Mandelstam matrix.
All signed Mandelstam regions ${\cal M}_{n,\sigma,r}$ are the same
up to sign changes. 

Indeed, given any Lorentzian matrix $S \in {\cal L}_{n,r}$, we obtain a Mandelstam matrix $S' \in {\cal M}_{n,\sigma,r}$ by conjugating $S$ with the matrix $J = {\rm diag}(\sigma_1,\sigma_2,\ldots,\sigma_n)$. The map $S \mapsto J S J$ is a linear isomorphism, and so
 ${\cal M}_{n,\sigma,r}$ is homeomorphic to  ${\cal L}_{n,r}$. 
The individual strata can have complicated topology in general, but the full region where we allow any rank is well-behaved.
Br\"anden \cite{Branden} proved that ${\mathcal{L}}_{n,\leq n}$ s a topological ball.
It has a decomposition by polymatroids, as shown in general by Br\"anden and Huh \cite{LorentzPol} and explained in more detail in Baker et al.~\cite{BHKL}.
Corollary \ref{cor:signs} says that ${\mathcal{M}}_{n,\leq n}$ is the union of $2^{n-1}$ such~balls.

\smallskip

In physics, a particle with momentum vector $p^{(i)}$ is said to have {\em mass} $\,m_i \geq 0\,$ if
\begin{equation}
s_{ii} \,\,=\,\, p^{(i)} \cdot p^{(i)} \,\,=\,\, (m_i)^2. 
\end{equation}
A particle with mass $m_i > 0$ is called \emph{massive}, and a particle with mass $m_i = 0$ is called \emph{massless}. 
The mass of a particle is a fixed constant. 
Thus, the momentum vector $p$ of a particle with mass $m > 0$ lies on the \emph{mass shell hyperboloid} given by $p \cdot p = m^2 > 0$. 

The real part of this hyperboloid is disconnected with two components, depending on the sign of $p_0$; see Figure \ref{fig:shell}. 
A massless momentum vector $p$ lies on the \emph{light cone}: $p \cdot p \,\,=\,\, 0.$

\medskip
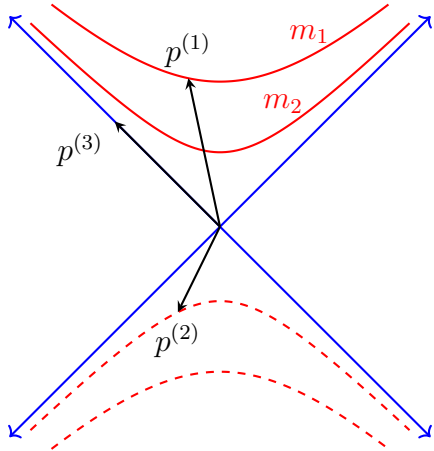
\begin{figure}[h]
\centering
\begin{tikzpicture}[scale=1.4]
  \draw[thick,blue,<->] (-2,-2) -- (2,2);   
  \draw[thick,blue,<->] (2,-2) -- (-2,2);  
  \draw[red,thick,domain=-1.8:1.8,samples=100] plot (\x, {sqrt(0.5+\x*\x)});
    \draw[red,thick,domain=-1.6:1.6,samples=100] plot (\x, {sqrt(1.9+\x*\x)});
      \draw[red,thick,dashed,domain=-1.8:1.8,samples=100] plot (\x, {-sqrt(0.5+\x*\x)});
    \draw[red,thick,dashed,domain=-1.6:1.6,samples=100] plot (\x, {-sqrt(1.9+\x*\x)});
    \node[red] at (0.85,1.83) {$m_1$};
     \node[red] at (0.6,1.1677) {$m_2$};
    \draw[thick,black,->, >=stealth] (0,0) -- (-0.3,1.4106) node[above]{$p^{(1)}$};
    \draw[thick,black,->, >=stealth] (0,0) -- (-0.4,-0.8124) node[below]{$p^{(2)}$};
    \draw[thick,black,->, >=stealth] (0,0) -- (-1,1) node[below left]{$p^{(3)}$};
\end{tikzpicture}
\caption{The light cone (blue) and the mass shells (red) for two given masses.}
\label{fig:shell}
\end{figure}

The mass shell hyperboloids are contained inside the two nappes of this cone: the upper nappe, $p_0 > 0$, and the lower nappe, $p_0 < 0$. In terms of this light cone (Figure \ref{fig:shell}), a Mandelstam matrix $S \in {\cal M}_{n,\sigma,r}$ is the Gram matrix of $n$ vectors $p^{(i)}$ that lie either inside ($p\cdot p > 0$) or on ($p\cdot p = 0$) the light cone. The entries of the sign vector $\sigma$ record which $p^{(i)}$ are in the upper nappe ($\sigma_i > 0$), and which in the lower nappe ($\sigma_i < 0$) of the light cone. 

Given that the masses $m_i$ are fixed quantities, we are motivated to study
the subsets of Mandelstam regions $ {\cal M}_{n,r}$ 
where each $s_{ii} = m_i^2$ is fixed to some non-negative value. Fixing $s_{ii} = m_i^2$ is known as the \emph{on shell condition} for a particle of mass $m_i$.  Most of this paper is devoted to
 particles that are massless ($m_i = 0$).
 Kinematic stratifications for massive particles are discussed in  Section \ref{sec6}.
The distinction is illustrated in Figure \ref{fig:mand}.

\section{Massless Particles}
\label{sec3}

Henceforth, we require the $n$ particles to be massless. 
The {\em massless Mandelstam region} $\mathcal{M}^0_{n,r}$ is
the semialgebraic set of Mandelstam matrices $S \in {\cal M}_{n,r}$ with zeros on the diagonal (i.e. $s_{11}=\cdots =s_{nn} = 0$). 
The {\em massless Lorentzian region}  $\mathcal{L}^0_{n,r}$ is
the intersection of $\mathcal{M}^0_{n,r}$ with the
non-negative orthant $(\R_{\geq 0})^{\binom{n}{2}}$. A matrix $S \in {\cal L}^0_{n,r}$ 
represents a multiaffine Lorentzian quadratic form. 
In this section we study the sign stratifications
of both $\mathcal{M}^0_{n,r}$ and $\mathcal{L}^0_{n,r}$.

Recall, from Lemma \ref{lem:mand}, that the principal minors of a Mandelstam matrix $S$ satisfy
the inequalities in \eqref{eq:minors}. Let us examine these inequalities upon restricting to the massless Mandelstam region. We have $s_{ii} = 0$ for the smallest minors, and the $2\times 2$ principal minors are ${\rm det}(S_{\{i,j\}}) = - s_{ij}^2 \leq 0$ for all pairs $\{i,j\}$. So these small minors satisfy Lemma \ref{lem:mand} trivially. However, for the $3\times 3$ principal minors 
of a massless Mandelstam matrix, we have
\begin{equation}
\label{eq:triples}
{\rm det}(S_{\{i,j,k\}}) \,\, = \,\, s_{ij} s_{ik} s_{jk} \,\, \geq \,\, 0.
\end{equation}
This  puts the same condition on the signs of off-diagonal entries as in \eqref{eq:signcond}. Moreover, for each quadruple $I = \{i,j,k,l\}$, the following quartic polynomial must be non-positive:
\begin{equation}
\label{eq:quadruples}
{\rm det}(S_{\{i,j,k,l\}}) \,\,=\,\,
s_{ij}^2 s_{kl}^2 \,+ \,
s_{ik}^2 s_{jl}^2 \,+\,
s_{il}^2 s_{jk}^2 
\,\,-\,\,2 \cdot \bigl(s_{ij}s_{ik} s_{jl} s_{kl} + s_{ij} s_{il} s_{jk} s_{kl}+s_{ik} s_{il} s_{jk} s_{jl} \bigr).
\end{equation}

If we pass to square roots, by setting $p_{ij} = \sqrt{s_{ij}},\ldots,p_{kl}=\sqrt{s_{kl}}$,
then (\ref{eq:quadruples}) factors:
\begin{equation}
\label{eq:factorsasfollows}
 \begin{matrix}
{\rm det}(S_{\{i,j,k,l\}}) &  = & 
    ( p_{ij} p_{kl} +p_{ik} p_{jl} +p_{il} p_{jk})
   ( -p_{ij} p_{kl} -p_{ik} p_{jl} +p_{il} p_{jk}) \\ & &  \,
   ( -p_{ij} p_{kl} +p_{ik} p_{jl} -p_{il} p_{jk})
   ( p_{ij} p_{kl} -p_{ik} p_{jl} -p_{il} p_{jk})
     . \end{matrix} 
\end{equation}     
The quartic (\ref{eq:quadruples}) is the squared version of the {\em Pl\"ucker quadric}, which is known as the {\em Schouten identity} in physics. We refer to the study of the {\em squared Grassmannian} in \cite[Section 3]{DFRS}. 

This observation guides us to the connection with matroid theory. 
We encounter the matroid decomposition of the space of multiaffine Lorentzian polynomials, due to
Br\"anden and Huh \cite{LorentzPol}, but with signs, and restricted to
matroids of rank two. All matroids in this paper have rank two. From now on, we
use the term ``matroid'' to mean ``rank two matroid.''

For us,  a {\em matroid}  on  $[n] = \{1,\ldots,n\}$ is a partition 
 $P = P_1 \sqcup P_2 \sqcup \cdots \sqcup P_m$ of
 a subset of $ [n]$ with $m \geq 2$. The {\em bases} of $P$ are the pairs $\{u,v\}$
 where $u \in P_i$ and $v \in P_j$ for~$i \not= j$.
 The elements in $[n] \backslash P$ are called {\em loops}.
  The matroid $P$ has $m$ parts $P_1,\ldots,P_m$, and it has $l = n-|P|$
 loops. 
The {\em uniform matroid} $U_n$ is the  partition of $P = [n]$
into $n$ singletons~$P_i = \{i\}$.

Fix a sign vector $\sigma \in \{-,+\}^n$. We identify $\sigma$ with its negation $-\sigma$. We call the pair $(P,\sigma)$ a {\em signed matroid}. Let $\mathcal{M}^0_{P,\sigma,r}$ be the subset
of the massless Mandelstam region~$\mathcal{M}^0_{n,r}$ defined by $ \,{\rm sign}(s_{ij}) =  \sigma_i \sigma_j \, $ if $\,\{i,j\}$ is a basis of $P$,
and $s_{ij} = 0\,$ if $\,\{i,j\}$ is not a basis of $P$.

\smallskip

The following theorem on the {\em kinematic stratification}  is the main result in this section.

\begin{theorem}\label{thm:massless}
Fix $r \geq 1$. The massless Mandelstam region is the disjoint union 
\begin{equation}\label{eq:Mstrata}
   \mathcal{M}^0_{n,r} \,\,=\,\, \bigsqcup_{P,\sigma} \,\mathcal{M}^0_{P,\sigma,r}
\end{equation}
where $(P, \sigma)$ runs over all signed matroids on $[n]$.
The {\em kinematic stratum} $\mathcal{M}^0_{P,\sigma,r}$ is non-empty if and only if
 $\,3 \leq r \leq m$ or $r=m=2$. If this holds, the dimension of the stratum~is
 \begin{equation}
 \label{eq:dimformula} {\rm dim} \bigl( \mathcal{M}^0_{P,\sigma,r} \bigr) \,\, = \,\, m(r-2)+n-l-\binom{r}{2}. 
 \end{equation}
\end{theorem}

\begin{proof}
We  decompose $\mathcal{M}^0_{n,r}$ by recording 
the sign matrix ${\rm sign}(S) = [{\rm sign}(s_{ij})] \in \{-,0,+\}^{n \times n}$
for each $S = [s_{ij}]$. If there are no zeros outside of the diagonal in the sign matrix then
$S \in \mathcal{M}^0_{U_n,\sigma, r}$ where $U_n$ is the uniform matroid, and $\sigma_i = {\rm sign}(\lambda_i)$
for the multipliers $\lambda_i$ in (\ref{eq:plambda}).

Suppose now that $S$ is a Mandelstam matrix which has some zero
off-diagonal entries.
We associate a matroid $P$ to $S$ as follows.
If $\lambda_i= 0$ then $i$ is a loop.
Otherwise, suppose $\{i,j\}$ is a pair of non-loops
 with $s_{ij} = 0$. Substituting it into the quartic in (\ref{eq:quadruples}), we find
\begin{equation}
\label{eq:blocksrankone}
{\rm det}(S_{\{i,j,k,l\}}) \,\, = \,\, (s_{ik} s_{jl} - s_{il} s_{jk})^2 \,\, \leq \,\, 0 ,
\quad \hbox{and hence} \quad s_{ik} s_{jl} = s_{il} s_{jk}.
\end{equation}
Thus, if also $s_{ik} = 0$, then either $s_{il} = 0$ or $s_{jk} = 0$, for all $l$.
Since $i$ is not a loop, there exists an $l$ such that $s_{il} \not= 0$,
and hence $s_{jk}=0$. Thus the relation given by the zeros of $S$
is an equivalence relation on the non-loops. In other words,
we obtain a matroid $P$ whose bases are the pairs $\{i,j\}$ with $s_{ij} \not= 0$.
 As before, we choose $\sigma$ to be the sign vector of $\lambda$. 
 
Now, given a signed matroid $(P,\sigma)$ we want to check when its stratum is non-empty. 
The condition  $r \leq m$ is necessary 
because the factorization of $S$ in (\ref{eq:gram}) has
only $m$ distinct momentum vectors $p^{(i)}$ up to scaling.
But, they span a space of dimension $r$, so we need $m \geq r$ vectors.
If $r=2$ then the light cone consists of two lines, so there are
only $m=2$ distinct momentum vectors $p^{(i)}$ up to scaling.
The case $r=1$ is impossible because no symmetric matrix
of rank one can have zeros on the diagonal. To show that the stated condition is sufficient, we
choose vectors $p^{(1)}, \ldots, p^{(n)}$ on the light cone such that
$p^{(i)} = 0$ if and only if $i$ is a loop in $P$, and
$p^{(i)}$ and $p^{(j)}$ are parallel 
if and only if $i$ and $j$ are parallel in $P$. 
For any $r$ between $3$ and $m$, we can select generic
configurations with this property which span a subspace $\R^{r}$ of $\R^{1+d}$.
For such a configuration, their Gram matrix is a point in
$ \mathcal{M}^0_{P,\sigma,r}$. For $r=m=2$ it suffices to pick two non-parallel vectors on the light cone. 

It remains to prove the dimension formula (\ref{eq:dimformula}).
For this, we observe that each matrix $S$ in $\mathcal{M}^0_{P,\sigma,r}$ has the following structure.
If $i$ is a loop of $P$ then the $i$th row and column are zero.
There is a diagonal block of zeros for each part of parallel elements in $P$.
All off-diagonal blocks are matrices of rank $\leq 1$, by (\ref{eq:blocksrankone}).
Let \underbar{$P$} denote the simple matroid underlying $P$. That is, \underbar{$P$} is the rank $2$ matroid 
obtained by removing all loops and keeping one element per parallelism class from $P$. Since all our matroids have rank two,  \underbar{$P$} is
simply the uniform matroid $U_m$, with one element for each part of $P$.

Every point in ${\cal L}_{\text{\underbar{$P$}},r}$ is
a rank $r$ Mandelstam matrix $T = [t_{\mu,\nu}]$ of size $m \times m$,
with rows and columns indexed by the parts $P_\mu, P_\nu$ of $P$.
From this we obtain the matrices $S$ in $\mathcal{M}^0_{P,\sigma,r}$ by
setting $s_{ij} = t_{\mu,\nu} \lambda_i \lambda_j$
for $i \in P_\mu$ and $j \in P_\nu$, where $\lambda \in \R^P \simeq \R^{n-l}$.

The space of $m\times m$ symmetric matrices  of rank $r$, with zeros on the diagonal, has dimension $m(r-1)-\binom{r}{2}$;
see e.g.~\cite[Theorem 6.1]{GramMatricesIsotropic}.  These are our degrees of freedom for choosing $T$.
When passing from $P$ to  \underbar{$P$},
we put together all parallel vectors in one vector.
So, we must enlarge this number by one dimension for
each multiplier $\lambda_i$ attached to the remaining $n-l-m$ non-loop momentum vectors.
This yields our dimension formula (\ref{eq:dimformula}).
\end{proof}

We now derive an explicit formula  for the number of kinematic strata 
of any given dimension $d$ in $\mathcal{M}^0_{n,r}$.
The aim is to count  signed matroids $P$ which have $m$ 
parts, subject to requiring that $\mathcal{M}^0_{P,\sigma,r}$
 is non-empty and has  dimension $d$.
By (\ref{eq:dimformula}),  the number of loops is 
 $$ l\,\,:=\,\,m(r-2)+n-\binom{r}{2}-d. $$
 The number of parts, $m$, in the matroid $P$ satisfies the following lower and upper bounds:
 $$ r\,\,\leq \,\,m \,\,\leq \,\,\frac{1}{r-1} \Bigl(d+\binom{r}{2}\Bigr) \,=:\,M. $$
  Moreover, for $r=2$, we must have $m=M=2$. 
  Our considerations imply the following formulas for the number of 
  strata by dimension.
 We write $\St{n-l}{m}$ for the {\em Stirling number of the second kind}.
  This is  the number of
    partitions of the set $[n-l]$ into exactly $m$ parts.
 
 \begin{corollary}\label{cor:countingMassless}
 The number of kinematic strata $\mathcal{M}^0_{P,\sigma,r}$ of dimension $d$
 in the Mandelstam region $\mathcal{M}^0_{n,r}$ 
  is given, for a fixed sign vector $\sigma$ or for all possible sign vectors, respectively, by
    \[
        \sum_{m=r}^{M}\binom{n}{l}\St{n-l}{m} \qquad \text{ and } \qquad 
        \textcolor{LightGreen}{\sum_{m=r}^{M}2^{n-l-1}\binom{n}{l}\St{n-l}{m}}.
    \]
\end{corollary}    
    
\begin{example}  \label{ex:counting}
        The numbers of strata for $n=4,5$ are given in two tables. Rows 
        are indexed by         $1\leq d \leq \binom{n}{2}$ and columns
        are indexed          by $2\leq r \leq n$.
                     The numbers in black are for 
                     the Lorentzian region $\mathcal{L}^0_{n,r}$ and the numbers in \textcolor{LightGreen}{green} are for 
                     the Mandelstam region $\mathcal{M}^0_{n,r}$.
        \begin{table}[h]            \begin{subtable}[t]{0.49\textwidth}
                \centering \scalebox{0.8}{
                \begin{tabular}{c|cc|cc|cc}
                    $\mathbf{d}$  / $\mathbf{r}$ & \multicolumn{2}{c|}{$\mathbf{2}$} & \multicolumn{2}{c|}{$\mathbf{3}$} & \multicolumn{2}{c}{$\mathbf{4}$} \\
                    \hline
                    $\mathbf{1}$ & $6$ &\textcolor{LightGreen}{$12$} &&&&\\
                    $\mathbf{2}$ & $12$ &\textcolor{LightGreen}{$48$} &&&&\\
                    $\mathbf{3}$ & $7$ &\textcolor{LightGreen}{$56$} & $4$ &\textcolor{LightGreen}{$16$} &&\\
                    $\mathbf{4}$ &&& $6$ &\textcolor{LightGreen}{$48$} &&\\
                    $\mathbf{5}$ &&& $1$ &\textcolor{LightGreen}{$8$} &&\\
                    $\mathbf{6}$ &&&&& $1$ &\textcolor{LightGreen}{$8$}
                \end{tabular}
                }
                \caption{$n=4$}
            \end{subtable}
            \begin{subtable}[t]{0.49\textwidth}
                \centering     \scalebox{0.8}{
                \begin{tabular}{c|cc|cc|cc|cc}
                    $\mathbf{d}$  / $\mathbf{r}$ & \multicolumn{2}{c|}{$\mathbf{2}$} & \multicolumn{2}{c|}{$\mathbf{3}$} & \multicolumn{2}{c|}{$\mathbf{4}$} & \multicolumn{2}{c}{$\mathbf{5}$} \\
                    \hline
                    $\mathbf{1}$ & $10$ &\textcolor{LightGreen}{$20$} &&&&&&\\
                    $\mathbf{2}$ & $30$ &\textcolor{LightGreen}{$120$} &&&&&&\\
                    $\mathbf{3}$ & $35$ &\textcolor{LightGreen}{$280$} & $10$ &\textcolor{LightGreen}{$40$} &&&&\\
                    $\mathbf{4}$ & $15$ &\textcolor{LightGreen}{$240$}& $30$ &\textcolor{LightGreen}{$240$} &&&&\\
                    $\mathbf{5}$ &&& $30$ &\textcolor{LightGreen}{$440$} &&&&\\
                    $\mathbf{6}$ &&& $10$ &\textcolor{LightGreen}{$160$} & $5$ &\textcolor{LightGreen}{$40$} &&\\
                    $\mathbf{7}$ &&&$1$ &\textcolor{LightGreen}{$16$}& $10$ &\textcolor{LightGreen}{$160$} &&\\
                    $\mathbf{8}$ &&&&&&&&\\
                    $\mathbf{9}$ &&&&& $1$ &\textcolor{LightGreen}{$16$} &&\\
                    $\mathbf{10}$ &&&&&&& $1$ &\textcolor{LightGreen}{$16$}
                \end{tabular} }
                \caption{$n=5$}
            \end{subtable}
            \caption{Counting kinematic strata.}
            \label{tb:countingKinematicStrata}
        \end{table}

    \end{example}
    
The set of all matroids on $[n]$ is a partially ordered set (poset).
In this poset, we have
$P \leq P'$ if every~loop of $P'$ is a loop in $P$, and
the partition $P'$ refines the partition $P$.
For this refinement, one removes loops of $P$ that are non-loops in $P'$.
The order relation corresponds to containment of matroid polytopes,
which, for rank $2$ matroids, is equivalent to the one used in \cite{BHKL}.
This poset structure extends naturally to signed matroids:
we have $(P,\sigma) \leq (P',\sigma')$ if and only if $P \leq P'$ and
  $\sigma=\sigma'$ for all non-loops of $P$. 

For any fixed rank $r$, we consider the restriction of this
poset to signed matroids $(P,\sigma)$ for which 
$\mathcal{M}^0_{P,\sigma,r}$ is non-empty.
In Section \ref{sec4} we shall see that this subposet 
is precisely the incidence relation among
the closures of the kinematic strata in the Mandelstam region $\mathcal{M}^0_{n,r}$.
We conclude this section by offering a preview of the $n=4$ strata in Table \ref{tb:countingKinematicStrata} above.

\begin{example}[$n=4, \sigma=+\!+\!+ +$]\label{ex:LorentzianPosetn4}
We discuss all kinematic strata of $\mathcal{L}^0_{4,r}$ for $r=4,3,2$. Our ambient space is $\R^6$.
 For $r=4$, the only stratum is $\mathcal{L}^0_{U_4,4}=\{(s_{ij})\in (\R_{>0})^6: \det(S)<0\}$.
  Indeed, if one matrix entry is $0$ then ${\rm det}(S)$ is a square and hence $\det(S)\geq 0$.
  For $r=3$ we restrict to the quartic hypersurface
    $\{\det(S)=0\}$ in $\R^6$. This has $11$ strata which come in three classes.
    On the $5$-dimensional stratum $\mathcal{L}^0_{U_4,3}$ all $s_{ij}$ are positive.
    This stratum has three connected components. On each component,
    precisely one of the  three last factors in (\ref{eq:factorsasfollows})    is zero.
    They meet along six $4$-dimensional strata, where precisely one $s_{ij}$ is zero.
On their boundaries, $i$ or $j$ can become a loop.
The  four $3$-dimensional strata are given by the four
matroids with $m=3$ and $l = 1$. Modulo scaling rows and columns of $S$,
the strata have dimensions $2,1,0$. Figure \ref{k4graph} gives an illustration.
The red square represents three connected components, each glued into
the complete graph $K_4$ along one of the three~$4$-cycles.

For $r=2$, there are $7=3+4$ strata of top dimension $3$,
given by the set partitions of~$[4]$ with two parts.
We obtain $12$ strata of dimension $2$, and $6$ strata of dimension $1$,
 by turning one or two of the elements in $[4]$ into loops.
Figure \ref{12-34poset} depicts one order ideal in our poset
for $r=2$, namely all strata that lie below the top stratum
$\mathcal{L}^0_{P,2}$ where $P = \{1,2\} \sqcup \{3,4\}$.
\end{example}
    
    \begin{figure}[h] \vspace{-0.2cm}
    \begin{subfigure}{0.44\textwidth}
        \centering 
        \includegraphics[width=\textwidth]{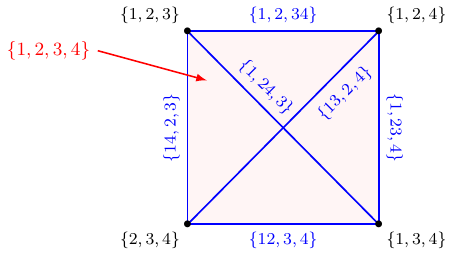} \vspace{-0.2cm}
            \caption{$\mathcal{L}_{4,3}$}
            \label{k4graph}
    \end{subfigure}\hfill%
    \begin{subfigure}{0.4\textwidth}
        \centering
        \includegraphics[width=\textwidth]{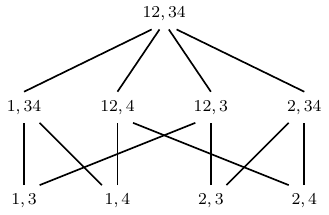} \vspace{-0.2cm}
        \caption{$\overline{\mathcal{L}_{(12,34),2}}$}
        \label{12-34poset}
    \end{subfigure}
	\caption{Posets of matroids.}
\end{figure}

\section{Inclusions and Topology}
\label{sec4}

We now turn to topological aspects of the kinematic stratifications 
for the Lorentzian region $\mathcal{L}^0_{n,r}$, resp.~the Mandelstam region $\mathcal{M}^0_{n,r}$.
We already know that the strata are indexed by the poset of matroids, resp.~signed matroids.
Our first result states that these posets indeed correspond to the inclusions among 
closures of the strata. The stratification of $\mathcal{L}^0_{n,r}$ is thus a
special case of the matroid decomposition for multiaffine Lorentzian polynomials of arbitrary
degree, which was studied recently by Baker, Huh, Kummer and Lorscheid~\cite{BHKL}.
However, in our case of quadratic polynomials, the stratification is nicer
than that for higher degree.

\begin{proposition} \label{prop:nicer}
Consider two non-empty kinematic strata
     $\mathcal{M}^0_{P,\sigma,r}$ and $\mathcal{M}^0_{P',\sigma',r}$. Then, we have
      $\,\mathcal{M}^0_{P,\sigma,r} \subseteq \overline{\mathcal{M}^0_{P',\sigma',r}} \,$ if and only if
    $\,\mathcal{M}^0_{P,\sigma,r} \cap \overline{\mathcal{M}^0_{P',\sigma',r}} \not= \emptyset\,$ 
    if and only if $\,(P,\sigma) \leq (P',\sigma')$.
\end{proposition}

\begin{proof}
The first statement clearly implies the second statement.
    The order relation $(P,\sigma) \leq (P',\sigma')$ among signed matroids
    means that the non-zero entries of a matrix in $\mathcal{M}^0_{P,\sigma,r}$ have the same sign pattern as the matrices in $\mathcal{M}^0_{P',\sigma',r}$. Since the signs of the entries $s_{ij}$ are weakly preserved
    under taking limits of matrices, the second statement implies the third statement.

Now suppose $\,(P,\sigma) \leq (P',\sigma')$,
    and let $S$ be any Mandelstam matrix in $ \mathcal{M}^0_{P,\sigma,r}$.
We must show that $S$ is the limit of a sequence of matrices in $\mathcal{M}^0_{P',\sigma',r}$.
Without loss of generality, we assume that $S$ has non-negative entries,
so we prove our claim for $ \mathcal{L}^0_{P,r}$ and $\mathcal{L}^0_{P',r}$.

Recall that $S$ is realized by a collection of vectors $p^{(i)}$ in the
positive nappe of the light cone. If $i$ is a loop in $P$ but not in $P'$, then
$p^{(i)} =0$. We can replace it by a nearby vector that is non-zero
and parallel to the other vectors in its part in $P'$.
Every part $P_j$ of $P$ is a union of parts of $P'$. We
perturb the vectors $p^{(k)}$ for $k \in P_j$ such that the
perturbed vectors are parallel according to the parts of $P'$.
The resulting matrix $S'$ can be chosen arbitrarily close to $S$.
In this manner, we construct a sequence which shows
that $S$ is in $\overline{\mathcal{L}^0_{P',r}}$.
\end{proof}

We now discuss the stratifications and the topology of the strata in more detail.
We  proceed by increasing rank, starting with $r=2$.
The Mandelstam region $\mathcal{M}^0_{n,2}$ consists of all 
symmetric $n \times n$ matrices of rank $2$
that have zeros on the diagonal. This region is~the variety in $\R^{\binom{n}{2}}$ defined by
the ideal of $3 \times 3$ minors in the polynomial ring with $\binom{n}{2} $ unknowns $s_{ij}$.
We know from \cite[Proposition 2.3]{GramMatricesIsotropic}
that this ideal is radical, and it is the intersection of 
$2^{n-1}-1$ toric ideals, one for each of the 
partitions $P = P_1 \sqcup P_2$ of $[n]$ into two 
non-empty~parts:
\begin{equation}
\label{eq:ideal} \bigcap_P \, \biggl(
\bigl\langle s_{ij} \,: \, i,j \in P_1 \,\,\,{\rm or} \,\,\,i,j \in P_2 \,\bigr\rangle
\,+\,
\bigl\langle\, s_{ik} s_{il} - s_{il} s_{jk} \,\,:\,\,
i,j \in P_1\, \,{\rm and} \,\, k,l \in P_2\, \bigr\rangle \biggr).
\end{equation}
Here $P$ runs over all loopless matroids 
on $[n]$ with precisely two parts.
The Mandelstam region  $\mathcal{M}^0_{n,2}$ is
the real algebraic variety defined by the determinental ideal in (\ref{eq:ideal}).
 The Lorentzian region 
$\mathcal{L}^0_{n,2}$ is the non-negative part of this affine variety.
The prime ideals above define the $2^{n-1}-1$
maximal strata $\mathcal{L}_{n,P}^0$.
The ideals of lower-dimensional strata 
are toric as well: they are ideal sums 
of subsets of the minimal primes in (\ref{eq:ideal}).
In particular, each stratum is a positive toric variety.
Namely, the stratum indexed by a matroid 
$P = P_1 \sqcup P_2$ 
is the positive part of a product of two projective
spaces $\PP^{|P_1|-1} \times \PP^{|P_2|-1}$.
Using the moment map, this is identified with the
corresponding product of two simplices, namely the polytope
\begin{equation}
\label{eq:toblerone} \Delta_{|P_1|-1}\, \times \,\Delta_{|P_2|-1} . 
\end{equation}
The kinematic strata in $\mathcal{M}^0_{n,2}$ arise from
these polytopes by choosing a sign vector~$\sigma$.
The points in the stratum $\mathcal{M}^0_{P,\sigma,2}$ are $n \times n$ matrices
of rank $2$ which  have a block structure:
\begin{equation} \label{eq:block1}
 S \,\,\,=\,\,\,
 \begin{bmatrix}
 0 & \Lambda & 0 \\
 \Lambda^T & 0 & 0 \\
 0 & 0 & 0 \end{bmatrix}.
\end{equation}
The three blocks are indexed by $P_1,P_2$, and the loops $[n] \backslash P$.
The matrix $\Lambda$ has rank one, and its entries
have fixed signs $+$ or $-$. In summary, and in view of Corollary \ref{cor:countingMassless}, we conclude:

\begin{corollary}\label{cor:rank2toric}
The Lorentzian region $\mathcal{L}^0_{n,2}$ 
has precisely  $(2^d-1) \binom{n}{d+1}$ strata
of dimension~$d$. Each of these is a cone over the polytope (\ref{eq:toblerone}), so
 $d = n-l-1= |P_1|+ |P_2|-1$.
The Mandelstam region $\mathcal{M}^0_{n,2}$ 
has  $(2^{2d}-2^d) \binom{n}{d+1}$ strata (\ref{eq:block1})
of dimension $d$.
Here $d= 1, 2,\ldots ,n{-}1$.
\end{corollary}

\begin{example}[$n=4,r=2$]
The ideal of $3 \times 3$ minors is the 
intersection  (\ref{eq:ideal})  of seven prime ideals in
$\R[s_{12},s_{13},s_{14},s_{23},s_{24},s_{34}]$.
Viewed projectively, the variety $\mathcal{M}^0_{4,2}$
is a surface, glued from four copies of $\PP^2$ and
three copies of $\PP^1 \times \PP^1$. The 
polyhedral surface corresponding to $\mathcal{L}^0_{4,2}$ is
glued from four triangles and three squares.
To visualize this surface, we label the six vertices of an octahedron
with $s_{12},\ldots,s_{34}$, we retain the $12$ edges, and
we glue  four of the facets to the three squares that span symmetry planes.
 This explains the f-vector $(6,12,7)$ we saw in Example \ref{ex:counting}.
 Figure \ref{12-34poset} shows the face poset for one of the three squares.
\end{example}

We now increase the rank by one, and we consider the case $r=3$.
These kinematic stratifications
exhibit a new phenomenon that is 
noteworthy: the strata can be disconnected.

\begin{proposition} \label{prop:disconnected}
For any signed matroid $(P,\sigma)$, with $m$ parts,
$\mathcal{M}^0_{P,\sigma,3}$
has $(m-1)!/2$ connected components.
In particular, $\mathcal{L}^0_{U_n,3}$ has $(n-1)!/2$ connected components.
\end{proposition}

\begin{proof}
We use the representation in (\ref{eq:plambda}),
with $d=r-1=2$. Up to scaling by multipliers $\lambda_1,\ldots,\lambda_n$, 
the $n$ momentum
vectors are $m$ points $x^{(i)}$ that lie on a unit circle $\mathbb{S}^1$.
The connected components of $\mathcal{M}^0_{P,\sigma,3}$ 
correspond to the combinatorially distinct ways of placing
 $m$ points on the circle $\mathbb{S}^1$.
The number of such cyclic arrangements is $(m-1)!/2$.
\end{proof}

We now state a general theorem, valid for any $r \geq 3$, about the
topology of the kinematic strata.
For~this, we soften our rank conditions and take the union of all strata
up to a given matrix rank $r$. The signed matroid
  $(P, \sigma)$ is still fixed. 
Thus we consider the enlarged strata
\begin{equation}
\label{eq:enlarged}
\mathcal{L}^0_{P,\leq r}\,=\, \bigsqcup_{r'\leq r} \mathcal{L}^0_{P,r'} \quad
 \text{ and }\quad \mathcal{M}^0_{P,\sigma,\leq r} \, = \,\bigsqcup_{r'\leq r} \mathcal{M}^0_{P,\sigma,r'}.
 \end{equation}
Theorem \ref{thm:massless} implies
$\mathcal{M}^0_{P,\sigma,3}=\mathcal{M}^0_{P,\sigma,\leq 3}$.
But, for $r \geq 4$, the unions in (\ref{eq:enlarged}) are non-trivial.

We saw in the proof of Lemma \ref{lem:gram}
that any Mandelstam matrix $S = [s_{ij}]$ in $\mathcal{M}^0_{n,\leq r}$~can be realized by $n$ multipliers
$\lambda_1,\ldots,\lambda_n$ plus a  configuration of
$m$ distinct points $x^{(\nu)}$ on the~sphere
$$ \mathbb{S}^{r-2} \,\, = \,\, \partial \mathbb{B}^{r-1} \,\, = \,\, \{ x \in \R^r \,: \,||x|| = 1 \}. $$
Namely, generalizing the block decomposition in (\ref{eq:block1}),
any Mandelstam matrix has the form
\begin{equation} \label{eq:block2}
S \,\,\, = \,\,\, \begin{small}
\begin{bmatrix} \,\,0 & \Lambda_{12} & \cdots & \Lambda_{1m} & 0 \\
\,\,\Lambda_{12}^T & 0 &  \cdots & \Lambda_{2m} & 0 \\
\,\, \vdots & \vdots & \ddots & \vdots & \vdots &  \\
\,\, \Lambda_{1m}^T & \Lambda_{2m}^T & \cdots & 0 & 0 \\
\,\, 0 & 0 & \cdots & 0 & 0 \\
\end{bmatrix}. \end{small}
\end{equation} 
Here the block $\Lambda_{\mu \nu}$ is a rank one matrix with non-zero entries.
The rows of $\Lambda_{\mu \nu}$ are labeled by $P_\mu$, the columns 
 of $\Lambda_{\mu \nu}$  are labeled by $P_\nu$,
and the entries are $s_{ij} = \lambda_i \lambda_j t_{\mu,\nu}$ for all 
$i \in P_\mu$ and $j \in P_\nu$. The Greek letters $\mu,\nu$
now index the parts of the matroid $P$.
Finally,
$$ t_{\mu,\nu} \,\, = \,\, 1 - \langle x^{(\mu)} , x^{(\nu)} \rangle, $$
which is at least $0$ and at most $2$.

Following \cite{FN, FZ}, we now introduce the  {\em ordered configuration space}
$F(\mathbb{S}^{r-2},m)$ for  $m$ distinct labeled points
$x^{(1)}, x^{(2)},\ldots, x^{(m)}$  on the $(r-2)$-sphere
$\mathbb{S}^{r-2}$. The orthogonal group ${\rm O}(r-1)$
acts naturally on this space, and we are interested in the quotient space,
denoted
\begin{equation} \label{eq:confspace} F(\mathbb{S}^{r-2},m)/{\rm O}(r-1). \end{equation}
We call this the {\em orbit configuration space} for $m$ points on $\mathbb{S}^{r-2}$.
The quantities $t_{\mu,\nu}$ furnish coordinates on that space.
They are invariant under ${\rm O}(r-1)$, and they characterize
the configuration uniquely up to rotations and up to the reflection given by negating all coordinates in $\R^{r-1}$. There is a natural map from
any Mandelstam stratum $\mathcal{M}^0_{P,\sigma,\leq r}$
to the orbit configuration space (\ref{eq:confspace}).
This map takes each Mandelstam matrix $S$ to the normalized matrix where each block
$\Lambda_{\mu,\nu}$ is simply the constant matrix with entry $t_{\mu,\nu}$.

The group $(\R_{> 0})^n$ acts on  $\mathcal{M}^0_{P,\sigma,\leq r}$
by scaling the rows and columns of $S$ with the multipliers
$\lambda_1,\ldots,\lambda_n$. The map above is the quotient map,
and it induces a homotopy equivalence. We have derived the following result
on the topology of the strata in the Mandelstam region.

\begin{theorem} \label{thm:topology} The kinematic stratum $\mathcal{M}^0_{P,\sigma,\leq r}$,
for a matroid $P$ with $m$ parts, is homotopy equivalent to 
the orbit configuration space $F(\mathbb{S}^{r-2},m)/{\rm O}(r-1)$
for $m$ points on the sphere.
\end{theorem}

This explains our findings for $r=3$ in  Proposition \ref{prop:disconnected}.
The orbit configuration space (\ref{eq:confspace}) for
$m$ points on the circle $\mathbb{S}^1$ is the union of
$(m-1)!/2$ contractible spaces, one for each of the
distinct arrangements of the $m$ points $x^{(\nu)}$ on the circle.
The kinematic stratum $\mathcal{M}^0_{P,\sigma,\leq 3}$ is homotopy
equivalent to that space: it has the homotopy type of $(m-1)!/2$ isolated points.

We now turn to the case of $r=4$, which is relevant
to describe the real world.
Every stratum $\mathcal{M}^0_{P,\sigma,\leq 4}$ has the homotopy type
of (\ref{eq:confspace}) for $m$ points on the $2$-dimensional sphere~$\mathbb{S}^{2}$.
The stratum is connected but its topology is very interesting.
Following Feichtner and Ziegler \cite[\S 2]{FZ}, we identify
$\mathbb{S}^2$ with the Riemann sphere $\CC\PP^1$.
Hence, $F(\mathbb{S}^{r-2},m)/{\rm SO}(r-1)$ is the space of $m$ points on the
complex projective line $\CC\PP^1$, which is the well-studied
moduli space $M_{0,m}(\CC)$. To obtain our orbit configuration space in \eqref{eq:confspace} we must further quotient by the action of $-1 \in {\rm O}(r-1)$, which corresponds to complex conjugation in $M_{0,m}(\CC)$.
Note that the moduli space $M_{0,m}(\CC)$ has complex dimension $m-3$, but
here we view it as a real manifold of dimension $2m-6$.
The subspace $M_{0,m}(\R)$ of its real points has real dimension $m-3$,
and it is the union of $(m-1)!/2$ curvy associahedra. This space is the
$r=3$ stratum discussed above. It is worthwhile to check the
dimensions against the formula in~(\ref{eq:dimformula}):
$$ {\rm dim}(\mathcal{M}^0_{P,\sigma,3}) \, - \, {\rm dim}(M_{0,m}(\R)) 
\,\,=\,\,
{\rm dim}(\mathcal{M}^0_{P,\sigma,\leq 4}) \, - \, {\rm dim}(M_{0,m}(\CC)) \,\,\, = \,\,\, n-l .$$
This equals the fiber dimension of the quotient by $(\R_{> 0})^n$, because $P$ has $l$ loops.

From Theorem 2.1 in the article \cite{FZ} 
we now conclude:

\begin{corollary} \label{cor:braid} Consider $n$ massless particles in $4$-dimensional spacetime,
and a signed matroid $(P, \sigma)$ as above.
The kinematic stratum $\mathcal{M}^0_{P,\sigma,\leq 4}$ is homotopy equivalent to
the~moduli space $M_{0,m}(\CC)$ modulo complex conjugation.
\end{corollary}

In physics, one is also interested in particles with spacetime dimension $r \geq 5$.
For these higher dimensions, Theorem \ref{thm:topology} relates the topology of the kinematic strata to the ordered configuration spaces $F(\mathbb{S}^{r-2},m)$. 
We refer to \cite[Corollary 5.3]{FN} for the homotopy type and to \cite[Theorem 5.1]{FZ} for the cohomology ring of these spaces.

\section{Momentum Conservation}
\label{sec5}

We now study the scenario of $n$ massless particles that satisfy momentum conservation; see
\eqref{eq:momcon} in the Introduction.
 The \emph{massless momentum conserving (MMC) region}~$\mathcal{C}^0_{n,r}$ is the semialgebraic set of Mandelstam matrices $S \in {\cal M}^0_{n,r}$ whose row sums and column sums are all zero. By Theorem \ref{thm:massless}, the MMC region admits a decomposition as the disjoint union
\begin{equation} \label{eq:Cstrati}
\mathcal{C}^0_{n,r} \,\, = \,\,\, \bigsqcup_{P,\sigma}\, \mathcal{C}^0_{P,\sigma,r},
\end{equation}
where $ \mathcal{C}^0_{P,\sigma,r}$ is the intersection of $ \mathcal{M}^0_{P,\sigma,r} $ with the
linear subspace   of $\R^{\binom{n}{2}}$ defined by \eqref{eq:momcon}.

\smallskip

In physics, the MMC region $\mathcal{C}^0_{n,r}$ comprises
the Gram matrices for all configurations of $n$ massless particles
that live in $r$-dimensional spacetime and satisfy momentum conservation. These matrices are relevant in the study of massless scattering amplitudes.
For determining which strata survive in the stratification (\ref{eq:Cstrati}),
an important role is played by the signs in~$\sigma$.
Indeed, the intersection of a Mandelstam stratum ${\cal M}^0_{P,\sigma,r}$ with the subspace \eqref{eq:momcon}
 may be empty. This depends
   on the choice of  sign vector $\sigma$. 
We call a signed matroid $(P,\sigma)$ \emph{$r$-momentum conserving} if ${\cal C}^0_{P,\sigma,r}$ is non-empty.
The following result characterizes which signed matroids satisfy this property and determines the dimension of its MMC stratum.

\begin{theorem}\label{thm:MatroidMomentumConservation}
$(P,\sigma)$ is $r$-momentum conserving if and only if the following conditions hold: 
\begin{enumerate}
    \item For $3\leq r < m$: there exist distinct $i,j,k,l$ in $[n]$, with $\sigma_i=\sigma_j = +$ and $\sigma_k=\sigma_l = -$, such that the restriction of 
    the matroid $P$ to $\{i,j,k,l\}$ is either $U_4$ or          $\{ik, jl\}$. \vspace{-0.1cm}
    \item For $2\leq r=m$: every part of $P$ has at least two elements with opposite signs.
\end{enumerate}
Moreover, if $(P,\sigma)$ is $r$-momentum conserving, then the dimension of its MMC stratum is
\begin{equation}
\label{eq:dim2}
\dim({\cal C}^0_{P,\sigma,r}) \,\,=\,\, (m - 1)(r-1) - \binom{r}{2} + (n-l-m) - 1.
\end{equation}
\end{theorem}
  
\begin{proof} 
We first check the conditions for ${\cal C}^0_{P,\sigma,r}$ to be non-empty. 
For $3 \leq m < r$, let us see why condition 1 is necessary. 
We first assume that all indices $i$ with $\sigma_i = +$
are parallel to each other in $P$. Fix such an index $i$.
Then $s_{ij} \leq 0$ for all $j$. There exists a non-loop $k$
which is not parallel to $i$. We have $s_{ik} < 0$.
These sign conditions imply $\sum_{j=1}^n s_{ij} < 0$, but this is 
a contradiction to momentum conservation (\ref{eq:momcon}). 
The other possible violation of condition  1 
 is that $m=3$, with at most one part
using both signs.  We can reduce this to the case
$n=4$, where $\{1,2\}$ is the unique parallel pair,
and $\sigma = (+,-,+,-)$. Then $s_{13}=-s_{23}$,
$s_{14}=-s_{24}$ and $s_{34} < 0$. From $s_{13}+s_{23} + s_{34} = 0$
and $s_{14}+s_{24}+s_{34}=0$, a contradiction is derived.

For the converse, suppose that $(P,\sigma)$ satisfies condition 1 in the theorem. 
We choose four distinct points $x^{(i)}, x^{(j)},x^{(k)},x^{(l)}$ on the sphere $\mathbb{S}^{r-2}$ such that
$\,x^{(i)} + x^{(j)}= x^{(k)} + x^{(l)}. $
By augmenting these points with a first coordinate $\pm 1$ depending on $\sigma$, 
we define $p^{(i)},p^{(j)},p^{(k)},p^{(l)}$ in $\R^{r}$.
These vectors lie on the light cone and satisfy momentum conservation. 
We next choose the remaining $n-4$ vectors $p$ to be very small but
to match $(P,\sigma)$ and so that all $n$ vectors span $\R^{r}$; 
we can do this for $3\leq r$ as in Theorem \ref{thm:massless}. 
Finally, we make small adjustments to $p^{(i)},p^{(j)},p^{(k)},p^{(l)}$ so that the
$n$ vectors sum to zero in $\R^{r}$. For $r<m$ we can choose the small vectors in such a way that the space
spanned by the $n$ vectors after the modifications is still of dimension $r$.
Then the resulting Mandelstam matrix (\ref{eq:gram}) lies in $\mathcal{C}^0_{P,\sigma,r}$.

For $r=m$, condition 2 is necessary because, summing all vectors in each part, we obtain a linear combination of $m$ independent vectors adding to zero. 
To see that it is also sufficient, we choose multipliers $\lambda_i$ such that the sum
over any of the parts is the zero vector.

\smallskip
The dimension count is similar to  Theorem \ref{thm:massless}. 
Let $S \in \mathcal{C}^0_{P,\sigma,r}$ with $P=P_1\sqcup \cdots \sqcup P_m$, and let $T$ be the
$m \times m$ Mandelstam matrix for the $m$ momentum vectors obtained
by summing the vectors in each part $P_i$.
Then, $T$ is a symmetric rank $r$ matrix with zeros on the diagonal and row/column sums equal to zero. 
The submatrix of $T$ given by eliminating the first row and first column still has
 rank $r$. This gives the $(m-1)(r-1)-\binom{r}{2}$ contribution to our dimension formula. 
Recall that we write momentum vectors as $p=\lambda(1,x)$ with $x \in \mathbb{S}^{r-2}$. 
Given $T$, the sum of the multipliers in each part $P_i$ is fixed.
Hence, we have only $n-l-m$ additional degrees of freedom to choose the multipliers.
Finally, the last $-1$ in our formula (\ref{eq:dim2}) comes from the fact that all multipliers sum to zero.
\end{proof}
  
To appreciate Theorem \ref{thm:MatroidMomentumConservation},
it is instructive to write down some signed matroids which fail to be $r$-momentum conserving. 
In the generic case, when $3\leq r < m$, there are only two disallowed situations:
either all positive elements are in the same part of the matroid $P$, or $\,P=P_1 \sqcup P_2 \sqcup P_3$ 
and all elements in $P_1$ are positive and all elements of $P_3$ are negative.

\begin{corollary} \label{cor:MMCtop}
  For $r\geq 3$, the MMC region has $2^{n-1}-n-1$
  full-dimensional strata $\mathcal{C}^0_{U_{n},\sigma,r}$.
  Each stratum corresponds to a sign vector $\sigma$ with $+$ and $-$ appearing at least twice.
\end{corollary}

We now count the MMC strata of a fixed dimension $d$. 
In light of Theorem \ref{thm:MatroidMomentumConservation}, we~set
\begin{equation}
l \,:= \, (m\!-\!1)(r\!-\!1) - \binom{r}{2} + (n\!-\!d\!-\!m)-1 \quad {\rm and} \quad
 M \, := \begin{cases}  \frac{d+r+\binom{r}{2}}{r-1} & {\rm if} \,\,\,\,r \geq 3, \smallskip \\
 \quad 2 & {\rm if} \,\,\,\,r = 2 . \\\end{cases}
 \end{equation}
The number of $r$-momentum conserving signed loopless matroids 
on $n$ elements equals
\begin{equation} \label{eq:doublesum} \quad
N_{n,m}^r \,\,=\,\, \frac{1}{2}\, \sum_{p=2}^{n-2} \sum_{\substack{a,b=2\\ a+b \geq m}}^{m}
 {n\choose p} \St{p}{a} \St{n-p}{b} \, \frac{a! \, b!}{(m-a)!(m-b)! (a+b-m)!} ~~\,
 {\rm for} \,\, r < m.
\end{equation}
Here $p$ denotes the number of indices $i$ with $\sigma_i = +1$.
We define $N_{n,m}^m$ by taking the outer sum in (\ref{eq:doublesum})  from $p=m$ to $p=n-m$.
In analogy to Corollary \ref{cor:countingMassless}, we can now derive:

\begin{corollary}
The number of strata $\,\mathcal{C}^0_{P,\sigma,r}$ of dimension $d$ in the MMC region  \textcolor{LightGreen}{$\,\mathcal{C}^0_{n,r}$} equals
    \[
         \textcolor{LightGreen}{\sum_{m=r}^{M}\binom{n}{l}N_{n-l,m}^r}.
    \]
A similar formula is available for counting the subset of strata that use a fixed sign vector $\sigma$.
\end{corollary}

The poset for the stratification of the MMC region is the restriction of the poset defined in Section \ref{sec3} for the Mandelstam region to $r$-momentum conserving signed matroids. 
The proof of Proposition \ref{prop:nicer}
 descends to the MMC region, showing this is indeed a stratification.
The poset governs when the closure of an MMC stratum contains lower dimensional strata. 
 
 \smallskip
 
 We conclude with a study of the MMC regions for $n=4,5$. 
 The numbers of MMC strata are  given in Table \ref{tab:MMCstrataCountex}.  They are 
  smaller than those for the Mandelstam strata in  Table~\ref{tb:countingKinematicStrata}.

\begin{table}[ht]
    \begin{subtable}[t]{0.49\textwidth}
    \centering
    \scalebox{0.8}{
    \begin{tabular}{c|cc|cc}
        $\mathbf{d}$  / $\mathbf{r}$ & \multicolumn{2}{c|}{$\mathbf{2}$} & \multicolumn{2}{c}{$\mathbf{3}$}\\
                    \hline
        $\mathbf{1}$ & $2$ & \textcolor{LightGreen}{$6$} && \\
        $\mathbf{2}$ &&& $1$ &\textcolor{LightGreen}{$3$}
    \end{tabular}}
    \caption{$n=4$}
            \label{tab:MMCstrataCountexA}

    \end{subtable}
    \begin{subtable}[t]{0.49\textwidth}
        
        \centering
        \scalebox{0.8}{
        \begin{tabular}{c|cc|cc|cc}
            $\mathbf{d}$  / $\mathbf{r}$ & \multicolumn{2}{c|}{$\mathbf{2}$} & \multicolumn{2}{c|}{$\mathbf{3}$} & \multicolumn{2}{c}{$\mathbf{4}$}\\
            \hline
            $\mathbf{1}$ &$6$&\textcolor{LightGreen}{$30$}&&&& \\
            $\mathbf{2}$ &$6$&\textcolor{LightGreen}{$60$}&$3$&\textcolor{LightGreen}{$15$}&&\\
            $\mathbf{3}$ &&&$9$&\textcolor{LightGreen}{$90$}&&\\
            $\mathbf{4}$ &&&$1$&\textcolor{LightGreen}{$10$}&&\\
            $\mathbf{5}$ &&&&&$1$&\textcolor{LightGreen}{$10$}
        \end{tabular}}
                \caption{$n=5$}
        \label{tab:MMCstrataCountexB}
    \end{subtable}
    \caption{Counting MMC strata.}
    \label{tab:MMCstrataCountex}
\end{table}

       \begin{example}[$n=4$] 
\label{ex:vier}
The regions counted in Table \ref{tab:MMCstrataCountexA} can be drawn in the 
$(x,y)$-plane, 
$$ S \,\, = \,\, \begin{small}
\begin{bmatrix}
0 & x & -x-y & y \\
x & 0 & y & -x-y \\
-x-y & y & 0 & x \\
y & -x-y & x & 0 \end{bmatrix}. \end{small}
$$
This matrix has rank $3$. Each triple $I$ in $ [4]= \{1,2,3,4\}$ yields the same inequality
$$ {\rm det}(S_I) \,=\, -2 xy (x+y) \,\,\geq \,\, 0. $$
This inequality defines the MMC region. It consists of
three closed convex cones in $\R^2$. We see that
$\mathcal{C}_{4,\leq 3}^0=\mathcal{C}^0_{4,3} \cup \mathcal{C}^0_{4,2}$ has nine MMC strata.
These are shown  in Figure~\ref{fig:MMCn4}, with red for $r=3$ and blue for $r=2$.
The uniform matroid $U_{4}$ contributes
    $\mathcal{C}^0_{U_{4},\sigma,3} =   \{x <0,y<0\}$ for $\sigma = (+,-,+,-)$,
    $\mathcal{C}^0_{U_{4},\sigma,3} =   \{x+y> 0,x<0\}$ for $\sigma = (+,-,-,+)$,
and $\mathcal{C}^0_{U_{4},\sigma,3} =   \{x+y>0,y<0\}$ for $\sigma = (+,+,-,-)$.
The matroid $P = \{12,34\}$ contributes the rays $\{x=0,y>0\}$ and $\{x=0,y<0\}$,
the matroid $\{13,24\}$ contributes the rays $\{x=-y > 0\}$ and $\{x=-y<0\}$, and
  $\{14,23\}$ contributes the rays $\{y=0,x>0\}$ and $\{y=0,x<0\}$.

\begin{figure}[h]
    \centering
\includegraphics[width=0.6\textwidth]{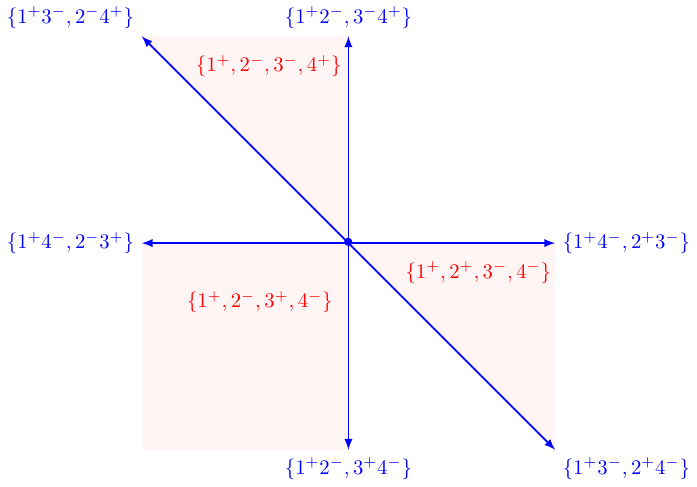}
    \caption{The ${\color{red} 3}+{\color{blue} 6}$ MMC strata for $n=4$.}
    \label{fig:MMCn4}
    \end{figure}
\end{example}

\begin{example}[$n=5$] 
\label{ex:fuenf}
Here $r \in \{2,3,4\}$.
We parametrize the $5$-dimensional space in (\ref{eq:momcon})~by
\begin{equation}
S \,\, = \,\, \begin{small}
\begin{bmatrix}
    0&a&\!\!-a-b+d&b-d-e&e\\
    a&0&b&\!\!-b-c+e&\!\!-a+c-e\\
    -a-b+d&b&0&c&a-c-d\\
   \, \,b-d-e&-b-c+e&c&0&d\\
    e&-a+c-e&a-c-d&d&0
\end{bmatrix}. \end{small}
\end{equation}
The ten matrix entries define a hyperplane arrangement $\mathcal{A}$ with $332$ regions in $\R^5$.
Only $10=2^{n-1}-n-1$ of the regions satisfy the inequalities (\ref{eq:triples}).
Thus  $U_{5}$ contributes $10$ MMC strata for ranks $r=3,4$,
as seen in  Table \ref{tab:MMCstrataCountexB}.
 These
are indexed by the rows in the  table:
$$ \begin{footnotesize} \begin{array}{c|cccccccccc}
    \sigma & s_{12} & s_{13} & s_{14} & s_{15} & s_{23} & s_{24} & s_{25} & s_{34} & s_{35} & s_{45}\\
    \hline 
    (-,-,+,+,+) &+&-&-&-&-&-&-&+&+&+\\
    (-,+,-,+,+) &-&+&-&-&-&+&+&-&-&+\\
    (-,+,+,-,+) &-&-&+&-&+&-&+&-&+&-\\
    (-,+,+,+,-) &-&-&-&+&+&+&-&+&-&-\\
    (+,-,-,+,+) &-&-&+&+&+&-&-&-&-&+\\
    (+,-,+,-,+) &-&+&-&+&-&+&-&-&+&-\\
    (+,-,-,+,+) &-&+&+&-&-&-&+&+&-&-\\
    (+,+,-,-,+) &+&-&-&+&-&-&+&+&-&-\\
    (+,+,-,+,-) &+&-&+&-&-&+&-&-&+&-\\
    (+,+,+,-,-) &+&+&-&-&+&-&-&-&-&+
    \end{array} 
    \end{footnotesize} $$
We fix one sign vector,    say $\sigma = (-,-,+,+,+)$.
The region of $\mathcal{A}$
is the cone $\mathcal{C}$ over the $4$-dimensional cyclic polytope $C(4,6)$,
with f-vector $(6,15,18,9)$.
This agrees with \cite[Example 5.2]{SHVarieties}.
The unique MMC stratum for $r=4$ is  $\mathcal{C}^0_{U_5,\sigma,4}$.
It is the open subset of $\mathcal{C}$ given~by
\begin{equation}\label{eq:quartic}
\begin{matrix}
    a^{2}b^{2}+b^{2}c^{2}+c^{2}d^{2}+d^{2}e^{2}+a^{2}e^{2}
     \,+\,2abcd+2abce+2abde    +2acde+2bcde \qquad\qquad   \qquad\\ 
     \qquad \qquad \qquad \qquad     \qquad \qquad\,\,
        -\,2ab^{2}c   -2bc^{2}d-2cd^{2}e -2ade^{2} -2a^{2}be\,\, < \,\, 0.
\end{matrix}        
\end{equation}
This is the determinant of any principal $4 \times 4 $ minor of $S$.
The quartic hypersurface in $\PP^4$ defined by (\ref{eq:quartic}) is known to algebraic geometers as
the {\em Igusa quartic}.
It separates $\mathcal{C}^0_{U_{5},\sigma,4}$ from its (much smaller) complement in $\mathcal{C}$.
The boundary is the top stratum $\mathcal{C}^0_{U_4,\sigma,3}$ for  $r=3$.

The strata $\mathcal{C}^0_{P,\sigma,3}$ and $\mathcal{C}^0_{P,\sigma,2}$
 for other matroids $P$
are given by Theorem \ref{thm:MatroidMomentumConservation}.
They correspond to faces of $C(4,6)$. No part of $P$ can contain $1$ and $2$
since these are the only negative particles in $\sigma$.
This mirrors the fact that $\{s_{12} = 0\}$ is not a facet of $C(4,6)$. We find it convenient to
draw the dual polytope  $C(4,6)^\circ$, which is the direct product of two triangles:
\begin{center}
    \includegraphics[width=0.45\textwidth]{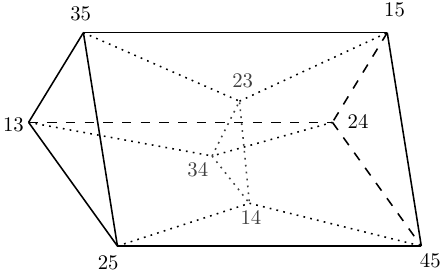}
\end{center}
Its vertices correspond to the nine 
$3$-dimensional MMC strata.
Finally, the MMC region for $r=3$ has three
$2$-dimensional strata   $\mathcal{C}^0_{P,\sigma,3}$.
These are indexed  by matroids $P$ with one loop, namely $3,4$ or $5$.
Geometrically, they
  correspond to three of the nine square faces of $C(4,6)^\circ$.
 
 The other six squares contribute $2$-dimensional
 MMC strata for $r=2$. For instance, the bottom face
 in our drawing of $C(4,6)^\circ$ gives
 $\{s_{13} = s_{24} = s_{25} = s_{45} = 0\}$.
 Finally, there are six $1$-dimensional strata
 $\mathcal{C}^0_{P,\sigma,2}$. These
 correspond to the six facets (``toblerone'') of $C(4,6)^\circ$.
 These reveal the $2$-momentum conserving $(P,\sigma)$ with
 $P = P_1 \sqcup P_2$ and $|P_1| = |P_2| = 2$.
\end{example}

\section{Stratifications and Scattering}
\label{sec6}

We conclude our study of  stratifications with some observations about how our results relate to the physics of scattering problems in 
quantum field theory \cite{BHPZ}. The object of such problems is to compute the amplitude, which is a function of the Mandelstam variables, $s_{ij}$, that form the entries of the symmetrix matrix $S$
in (\ref{eq:gram}). If the particles being scattered are massless, then the amplitude is a function on the 
MMC region ${\cal C}^0_{n,r}$, where $n$ is the number of particles.

We saw in  Section \ref{sec5} that ${\cal C}^0_{n,r}$ decomposes as the union of strata ${\cal C}^0_{P,\sigma,r}$
which are indexed by signed matroids $(P,\sigma)$.
There are $2^{n-1}-n-1$ full-dimensional strata, ${\cal C}^0_{U_n,\sigma,r}$, labelled only by the sign vector $\sigma$ (Corollary \ref{cor:MMCtop}). Each of these sign vectors corresponds to a distinct physical situation. Write $p^{(i)} = \sigma_i k^{(i)}$, where each $k^{(i)}$ is a future-pointing vector, $k_0^{(i)} > 0$, in the upper nappe of the light cone. Then the momentum conservation relation (\ref{eq:momcon}) reads
\[
\sum_{i :\sigma_i = +} \! k^{(i)} \,\,\,=\,\, \sum_{j : \sigma_j = -} \! k^{(j)}.
\]
This configuration describes the particles with $\sigma_i=+$ coming in from the past, scattering off each other, and then producing the particles with $\sigma_j=-$. For example, the sign vector $\sigma = (+,+,-,-,-,-)$ describes a reaction $1 + 2 \rightarrow 3 + 4 + 5 + 6$, that produces 
four particles from two, while the sign vector $\sigma = (+,+,+,-,-,-)$ describes a reaction $1+2+3 \rightarrow 4+5+6$.

Physicists conjecture that the amplitude can be analytically continued from the region with one sign vector, $\sigma$, to a region with a second sign vector, $\sigma'$, in such a way that the function takes a similar form on both regions. This is called \emph{crossing symmetry}. However, this is difficult to prove, because amplitudes have both poles and branching singularities as analytic functions of the $s_{ij}$, regarded as complex variables. In analyses of crossing symmetry, it is important to understand how the stratification of ${\cal C}^0_{n,r}$ that we have studied extends to
the space of
matrices $S$ with complex entries.  See \cite{Mizera} for a modern study of this problem.

Some of the singularities of amplitudes are captured by the stratifications we have described.
Each stratum ${\cal C}^0_{P,\sigma,r}$ is labeled by a signed matroid $(P,\sigma)$.
This partitions the particles and fixes a sign vector.
The boundaries of this stratum are labelled by $(P',\sigma') < (P,\sigma)$, and these can be produced from $(P,\sigma)$ in one of two ways. First, an entry $i$ of $P$ can become a loop in $P'$, which means that the corresponding vector $p^{(i)}$ becomes the zero vector. This is known as a \emph{soft limit}. Second, two entries $i,j$ of $P$, that are not parallel, can become parallel in $P'$. This is known as a \emph{collinear limit}. Amplitudes often have physically important divergences in these two types of limits. The different ways that nested divergences can arise is captured by chains in the poset of signed matroids that describes the stratification.

In addition to the particles that enter and exit a scattering process, quantum field theory allows for \emph{virtual particles} to arise as an intermediate step. For this reason, amplitudes are sometimes given by integrals over some virtual momentum vectors $\ell^{(i)}$.
  Gram matrices involving both the $p^{(j)}$ and these virtual $\ell^{(i)}$ arise in studies of these integrals and their singularities
\cite{Baikov, Henn}. Fixing the rank of these Gram matrices imposes constraints on their entries of the kind studied in this paper. However, these matrices are not Mandelstam matrices: not all diagonal entries are non-negative, because we allow $\ell^{(i)} \cdot \ell^{(i)} < 0$. Extending our analysis of stratifications to these non-Mandelstam regions will be an interesting problem.

The topology of the strata is studied in Theorem \ref{thm:topology}. 
When considering our customary
 $4$-dimensional spacetime (Corollary \ref{cor:braid}), the regions are
  related to Grassmannians via  \emph{spinor-helicity variables} \cite[Section~1.8]{BHPZ}. Here,
  a momentum vector in $\R^{1+3}$ is specified by a pair of complex vectors $\lambda, \widetilde{\lambda} \in \CC^2$. Following \'Elie Cartan, these are called \emph{spinors}, and they define representations of ${\rm SL}_2(\mathbb{C})$, the double cover of the Lorentz group ${\rm SO}(1,3)$. 
    One writes
\begin{equation}\label{eq:cartan}
s_{ij} \,\,=\,\, p^{(i)} \cdot p^{(j)} \,\,=\,\, [i\,j] \langle i\,j \rangle \,\, = \,\,\det\bigl[ \lambda^{(i)} \lambda^{(j)} \bigr]  \det\bigl[ \,\widetilde{\lambda}^{(i)} \widetilde{\lambda}^{(j)} \bigr].
\end{equation}
These $2 \times 2$ determinants are Pl\"ucker coordinates on ${\rm Gr}(2,n)\times {\rm Gr}(2,n)$.
We obtain our regions only if we impose appropriate reality conditions on the spinors. Namely, we set
\begin{equation} \label{eq:cartansign}
 \widetilde{\lambda}^{(i)} \,\,= \,\,\sigma_i \, \overline{\lambda^{(i)}},
\end{equation}
 where the bar denotes complex conjugation. Then the $s_{ij}$ are real valued and $\text{sgn}(s_{ij}) = \sigma_i \sigma_j$. 
The algebraic geometry behind   \eqref{eq:cartan}
  was studied in  \cite{SHVarieties}. The spinor-helicity variety ${\rm M}(2,n,0)$ (resp.~${\rm M}(2,n,2)$)
  is the Zariski closure of the MMC region $\mathcal{C}^0_{n,4}$ (resp.~$\mathcal{M}^0_{n,4}$).
   Note that, when restricted to $r=4$, our dimensions in (\ref{eq:dimformula}) and (\ref{eq:dim2}) agree with those in  \cite[eqn (32)]{SHVarieties}.

In this paper, we have focused on massless particles ($s_{ii} = 0$), such as gluons.
Our analysis also sets the stage for future work on
 kinematic regions for particles with non-zero masses. 
For any fixed vector ${\bf m} = (m_1,\ldots,m_n)$ of
non-negative masses, we can define analogous regions $\mathcal{M}^{\bf m}_{n,r}$,
$\mathcal{L}^{\bf m}_{n,r}$ and $\mathcal{C}^{\bf m}_{n,r}$.
These arise by restricting to Mandelstam matrices $S$ with
$s_{ii} = m_i^2$ for $i=1,\ldots,n$.
It would be interesting to extend the results of this paper to describe 
these semialgebraic sets and their strata.
In this direction, we give an example.


\begin{figure}[h]
  \centering
  \begin{subfigure}[c]{0.33\textwidth}
  	\includegraphics[width=\textwidth]{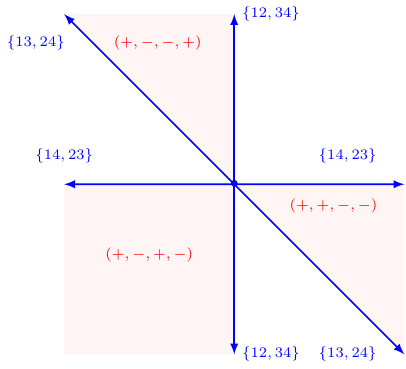}
    \caption{}
    \label{fig:m1}
  \end{subfigure}%
  \hfill %
  \begin{subfigure}[c]{0.33\textwidth}
  \centering
	\includegraphics[width=\textwidth]{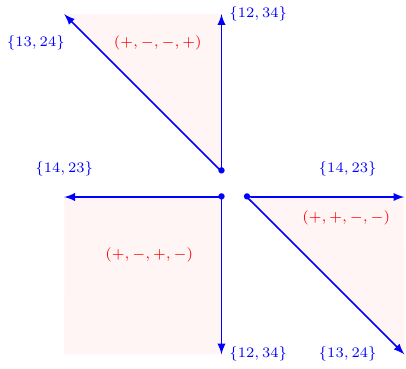}
    \caption{}
    \label{fig:m2}
  \end{subfigure}%
  \hfill%
  \begin{subfigure}[c]{0.33\textwidth}
  \centering
  	\includegraphics[width=\textwidth]{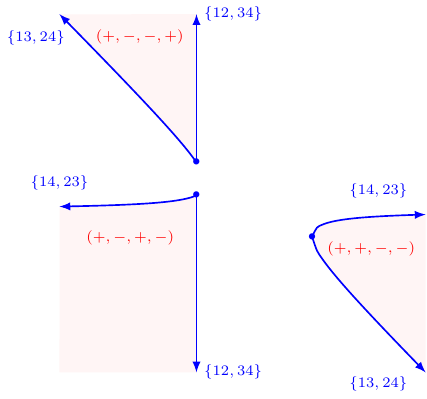}
    \caption{}
    \label{fig:m3}
  \end{subfigure}
  \caption{Regions for (a) massless, (b) equal masses, and (c) two unequal masses.}
  \label{fig:mand}
\end{figure}

For $n=4$ particles, take ${\bf m} = (\mu,\mu,m,m)$, for two masses $m > \mu >0$. We examine the region
$\mathcal{C}^{\bf m}_{4,3}$ by modifying Example  \ref{ex:vier}.
A Gram matrix $S$ in this region takes the form
$$
S \,\,=\,\,\begin{small} \begin{bmatrix} \mu^2 & x & -x-y-\mu^2 & y \\ x & \mu^2 & y & -x-y-\mu^2 \\ -x-y-\mu^2 & y & m^2 & \mu^2-m^2+x \\ y & -x-y-\mu^2 & \mu^2-m^2+x & m^2 \end{bmatrix}\!. \end{small}
$$
 The following inequalities for $ {\cal C}^{(\mu,\mu,m,m)}_{4,3}$ are seen from the $2 \times 2$-minors:
\begin{equation}
y^2 - m^2\mu^2 > 0, \quad (x+\mu^2)(x-2m^2+\mu^2) > 0, \quad (x+y)(x+y+2\mu^2)- \mu^2(m^2-\mu^2) > 0.
\end{equation}
These conditions exclude three strips that are parallel to the blue lines in Figure \ref{fig:m1}. The signs of the
 $3 \times 3$ minors furnish additional cubic inequalities. 
 In our example, with only two distinct masses,
 all $3\times 3$ minors of $S$ are equal and they factor.
  We obtain the condition
\begin{equation}\label{eq:massivecubic}
 \bigl(x+\mu^2\bigr) \bigl(2xy+2y^2+(m^2+\mu^2)x+2\mu^2y- m^2\mu^2+\mu^4 \bigr) \,<\, 0.
\end{equation}
The inequalities give three strata, shown in Figure \ref{fig:m3},
with linear and quadratic boundaries.
  In the limit $\mu,m \rightarrow 0$, the cubic in \eqref{eq:massivecubic} degenerates and we recover the cones of Figure~\ref{fig:m1}. 
  
  The massive case exhibits noteworthy novelties. Note that the stratum $\{1^+,2^+,3^-,4^-\}$ is bounded by the hyperbola in \eqref{eq:massivecubic}.
Whereas, the other two strata, $\{1^+,2^-,3^-,4^+\}$ and $\{1^+,2^-,3^+,4^-\}$, are bounded by both a line and a hyperbola. In physics, these strata correspond to the scattering of a $\mu$ particle and an $m$ particle. The $\{1^+,2^+,3^-,4^-\}$ stratum is different: it corresponds to two $\mu$ particles annihilating and producing two $m$ particles. This physical difference is reflected in the geometry of the strata for this region.

Remarkably, this very example was studied by Mandelstam himself,
in his 1958 article \cite{Mandelstam}.
His corresponding region is shown  in \cite[Figure 1]{Mandelstam}, and it matches
our Figure \ref{fig:m3}.  The study of the on-shell regions
$\mathcal{C}^{\bf m}_{n,r}$ for $n \geq 5$ 
will be an interesting subsequent research project.

\vfill
\pagebreak
\noindent {\bf Acknowledgement}: 
HF is supported by the U.S. Department of Energy (DE-SC0009988). This project was supported by the ERC (UNIVERSE PLUS, 101118787). 
$\!\!$ \begin{scriptsize}Views~and~opinions expressed 
are however those of the authors only and do not necessarily reflect those of the European Union or the European
Research Council Executive Agency. Neither the European Union nor the granting authority 
can be held responsible for them.
\end{scriptsize}

\bigskip \medskip

\noindent
\footnotesize {\bf Authors' addresses:}

\noindent Veronica Calvo Cortes,
MPI-MiS Leipzig \hfill {\tt veronica.calvo@mis.mpg.de}

\noindent Hadleigh Frost, IAS Princeton \hfill {\tt frost@ias.edu}

\noindent Bernd Sturmfels, MPI-MiS Leipzig \hfill {\tt bernd@mis.mpg.de}


\medskip \bigskip

\begin{thebibliography}{1}
\begin{small}
\setlength{\itemsep}{-0.4mm}


\bibitem{BHPZ}
Simon Badger, Johannes Henn, Jan Christoph Plefka, and Simone Zoia:
{\em Scattering Amplitudes in Quantum Field Theory},
Lecture Notes in Physics {\bf 1021}, Springer,  2024.

\bibitem{Baikov}
Pavel A.~Baikov:
{\em Explicit solutions of the multi-loop integral recurrence relations and its application},
Nuclear Instruments and Methods in Physics Research A
{\bf 389} (1997) 347--349.

\bibitem{Baker}
Matthew Baker: {\em
 Lorentzian polynomials {I}: {T}heory}, 2019,
 \url{https://mattbaker.blog/2019/08/30/lorentzian-polynomials/}.

\bibitem{BHKL}
Matthew Baker, June Huh, Mario Kummer, and Oliver Lorscheid:
{\em Lorentzian polynomials and matroids over triangular hyperfields 1: Topological aspects},
\href {https://arxiv.org/abs/2508.02907} {\path{arXiv:2508.02907}}.

\bibitem{Branden}
Petter Br{\"a}nd\'en:
{\em Spaces of {L}orentzian and real stable polynomials are {E}uclidean balls},
 Forum Math. Sigma {\bf 9} (2021) e73.

\bibitem{LorentzPol}
Petter Br\"and\'en and June Huh:
{\em Lorentzian polynomials},
 Ann.~Math.~{\bf 192} (2020) 821--891.

\bibitem{DFRS} Karel Devriendt, Hannah Friedman, Bernhard Reinke
and Bernd Sturmfels:
{\em The two lives of the Grassmannian}, Acta Universitatis Sapientiae Math  {\bf 17}, 8 (2025).
 
\bibitem{SHVarieties}
Yassine El~Maazouz, Ana\"elle Pfister, and Bernd Sturmfels:
{\em Spinor-helicity varieties}, SIAM Journal on Applied Algebra and Geometry {\bf 9} 4 (2025) 848--876.
 
 \bibitem{GramMatricesIsotropic}
 Yassine El~Maazouz, Bernd Sturmfels, and  Svala Sverrisd\'ottir:
{\em Gram matrices for isotropic vectors},
\href {https://arxiv.org/abs/2411.08624} {\path{arXiv:2411.08624}}.

\bibitem{FN}
Edward Fadell and Lee Neuwirth: {\em Configuration spaces},
Math.~Scand.~{\bf 10} (1962) 111-118.

\bibitem{FZ}
Eva Maria Feichtner and G\"unter M.~Ziegler:
{\em The integral cohomology algebras of ordered configuration spaces of spheres},
Documenta Mathematica {\bf 5} (2000) 115--139.

\bibitem{Cauchy}
Steve Fisk: {\em
A very short proof of {C}auchy's interlace theorem for eigenvalues of {H}ermitian matrices},
\href {https://arxiv.org/abs/math/0502408} {\path{arXiv:math/0502408}}.

\bibitem{Henn}
Johannes Henn, Antonela Matija{\v{s}}i{\'c}, Julian Miczajka, Tiziano Peraro,  Yingxuan Xu, and Yang Zhang:
{\em A computation of two-loop six-point Feynman integrals in dimensional regularization},
Journal of High Energy Physics {\bf 8} (2024) 1--38.

\bibitem{Mandelstam}
Stanley Mandelstam:
{\em Determination of the pion-nucleon scattering amplitude from dispersion relations and unitarity. General theory}, Physical Review {\bf 112} (1958) 1344--1360.

\bibitem{Mizera}
Sebastian Mizera: {\em Crossing symmetry in the planar limit},
Phys.~Rev.~D {\bf 104} (2021) 045003.

\end{small}
\end{thebibliography}
\end{document}